\documentclass[reqno,12pt]{amsart}
\usepackage{amsmath,amssymb,latexsym}
\usepackage{hyperref}
\usepackage{color}  
\usepackage[all]{xy}
\hypersetup{
   colorlinks=true, 
    linktoc=all,     
    linkcolor=blue,  
}

\newcommand{\Z}{{\mathbb{Z}}}
\newcommand{\Q}{{\mathbb{Q}}} 
\newcommand{\N}{{\mathbb{N}}}

\newcommand{\lra}{\longrightarrow}

\newcommand{\cyc}{\mathrm{cyc}}

\newcommand{\La}{\Lambda}

\newtheorem{theorem}{Theorem}
\newtheorem*{theorem*}{Theorem}

\newtheorem{rem}[theorem]{Remark}
\newtheorem*{rem*}{Remark}
\newtheorem{lemma}[theorem]{Lemma}

\newtheorem*{example}{Examples}
\newtheorem{claim}{Claim}
\newtheorem*{maintheorem}{Main Theorem}
\newtheorem*{preliminarytheorem}{Preliminary Theorem}
\newtheorem*{keylemma}{Key Lemma}

\author{Somnath Jha, Tadashi Ochiai, Gergely Z\'{a}br\'{a}di} \thanks{S. Jha gratefully acknowledges the support of JSPS postdoctoral fellowship and  DST INSPIRE faculty award  grant. T. Ochiai is partially supported for this work by KAKENHI (Grant-in-Aid for Exploratory Research: Grant Number 24654004, 
Grant-in-Aid for Scientific Research (B): Grant Number 26287005). G. Z\'abr\'adi was supported by a Hungarian OTKA Research grant K-100291 and by the J\'anos Bolyai Scholarship of the Hungarian Academy of Sciences.}\address{Department of Mathematics and Statistics, Indian Institute of Technology Kanpur,  Kanpur 208016, India} \email{jhasom@iitk.ac.in} \address{Department of Mathematics, Graduate School of Science, Osaka University, Machikaneyama 1-1, Toyonaka, Osaka 5600043, Japan}\email{ochiai@math.sci.osaka-u.ac.jp} 
\address{E\"{o}tv\"{o}s Lor\'{a}nd University, Mathematical Institute, 
Department of Algebra and Number Theory, 
1111 Budapest, Bertalan Lajos utca 11, Hungary}\email{zger@cs.elte.hu}
\begin{document}
\title{On twists of modules over non-commutative Iwasawa algebras}
\subjclass[2010]{11R23, 16S50}
\keywords{Selmer group, non-commutative Iwasawa theory}
\begin{abstract}  
It is well known that, for any finitely generated torsion module $M$ 
over the Iwasawa algebra $\Z_p [[\Gamma ]]$, where $\Gamma$ is isomorphic to $\Z_p$, 
there exists a continuous $p$-adic character $\rho$ of $\Gamma$ such that, 
for every open subgroup $U$ of $\Gamma$, the group of $U$-coinvariants $M(\rho)_U$ 
is finite; here $M(\rho)$ denotes the twist of $M$ by $\rho$. 
This twisting lemma was already applied to study various arithmetic properties of Selmer groups  and Galois cohomologies over a cyclotomic tower by Greenberg and Perrin-Riou.  
We prove a non commutative generalization of this twisting lemma replacing torsion modules over $\Z_p [[ \Gamma ]]$ by certain torsion modules 
over $\Z_p [[G]]$ with more general $p$-adic Lie group $G$. 
In a forthcoming article, this non-commutative twisting lemma will be applied to prove the functional equation of Selmer groups of general 
$p$-adic representations over certain $p$-adic Lie extensions.
\end{abstract}
\maketitle
\section*{Introduction}
Let us fix an odd prime $p$ throughout the paper. We denote by $\Gamma$ a $p$-Sylow subgroup of $\mathbb{Z}^\times_p$.  
For a compact $p$-adic Lie group $G$ and the ring $\mathcal{O}$ of integers of a finite extension of $\Q_p$,  
we denote the Iwasawa algebra $\mathcal{O}[[G]]$ of $G$ with coefficient in $\mathcal{O}$ by $\La_{\mathcal{O}}(G)$.  
\par 
In this article, we study $\La_{\mathcal{O}}(G)$-modules motivated by \cite{cfksv}. 
More precisely,  we study specializations of certain $\La_{\mathcal{O}}(G)$-modules by two-sided ideals of $\La_{\mathcal{O}}(G)$. 
Recall that the paper \cite{cfksv} establishes a reasonable setting of non-commutative Iwasawa theory in the following situation. 
\begin{enumerate}
\item[(G)] $G$ is a compact $p$-adic Lie group which has a closed normal subgroup $H$ such that $G/H$ is 
isomorphic to $\Gamma$. 
\end{enumerate}
According to the philosophy of \cite{cfksv}, for a reasonable ordinary $p$-adic representation $T$ of a number field $K$ 
and a pair of compact $p$-adic Lie groups $H \subset G$ satisfying the condition (G), the Pontryagin dual $\mathcal{S}^\vee_A$ 
of the Selmer group $\mathcal{S}_A$ of the Galois representation 
$A=T\otimes \Q_p /\Z_p$ over a Galois extension $K_\infty /K$ with $\mathrm{Gal} (K_\infty /K) \cong G$ seems to be a nice object. 
The $\La_{\mathcal{O}}(G)$-module $\mathcal{S}^\vee_A$  divided by the largest $p$-primary torsion subgroup $\mathcal{S}^\vee_A(p)$ is conjectured to belong to  
the category $\mathfrak n_H(G)$ which consists of finitely generated 
$\La_{\mathcal{O}}(G)$-modules $M$ such that $M$ is also finitely generated over $\La_{\mathcal{O}}(H)$. 
From such arithmetic background, we are led to study finitely generated $\La_{\mathcal{O}}(G)$-modules 
for a compact Lie group $G$ with $H \subset G $ satisfying the condition (G).  
\par 
On the other hand, for any open subgroup $U$ of $G$ and for any arithmetic module $\mathcal{S}^\vee_A$ 
as above, the largest $U$-coinvariant quotient $(\mathcal{S}^\vee_A)_U$ is expected to be related to 
the Selmer group of $A$ over a finite extention $K \subset K_\infty$ with $\mathrm{Gal} (K/\mathbb{Q})\cong G/U$.  
As remarked above, we have the following fact (Tw) when $G=\Gamma$ (i.e. when $H=1$) which was used quite usefully 
in the work of Greenberg \cite{gr} and Perrin-Riou \cite{pr}. 
\begin{enumerate}
\item[(Tw)] For any finitely generated torsion $\La_{\mathcal{O}}(\Gamma )$-module $M$, there exists 
a continuous character $\rho : \Gamma \longrightarrow \Z^\times_p$ such that the largest $U$-coinvariant quotient 
$(M\otimes_{\Z_p} \Z_p (\rho ))_U$ of $M\otimes_{\Z_p} \Z_p (\rho )$ is finite for every open subgroup $U$ of 
$\Gamma$ where $\Z_p (\rho )$ is a free $\mathbb{Z}_p$-module of rank one on which $\Gamma$ acts through the 
character $\Gamma \overset{\rho}{\rightarrow} \Z^\times_p$. 
\end{enumerate}
We call such a statement (Tw) twisting lemma. In this commutative situation of $G=\Gamma$, twisting lemma 
is proved in a quite elementary way. For example, we consider the characteristic ideal $\mathrm{char}_{\mathcal{O}[[\Gamma ]]} M$.
If we take a $\rho$ so that the values $\rho (\gamma)^{-1}\zeta_{p^n} -1$ do not coincide with any roots of the distinguished polynomial 
associated to $\mathrm{char}_{\mathcal{O}[[\Gamma ]]} M$ when natural numbers $n$ and $p^n$-th roots of unity $\zeta_{p^n}$ vary, twisting lemma is known to hold.   
\par 
If we have a twisting lemma in non-commutative setting, it seems quite useful for some arithmetic applications for 
non-commutative Iwasawa theory. On the other hand, for a non-commutative $G$, it was not clear what to do 
to prove twisting lemma because we can not talk about ``roots of characteristic polynomials" as we did in commutative setting. 
We finally succeeded in proving twisting lemma. We will state the result below. 
\par 
For a $\La_{\mathcal{O}}(G)$-module $M$ and a continuous character $\rho : \Gamma \longrightarrow \Z^\times_p$, 
we denote by $M (\rho)$ the $\La_{\mathcal{O}}(G)$-module $M \otimes_{\Z_p} \Z_p (\rho )$ with diagonal $G$-action. 
Our main result of this paper is the following theorem.
\begin{maintheorem}\label{mainin} 
Let $G$ be a compact $p$-adic Lie group and let $H$ 
be its closed normal subgroup such that $G/H$ is isomorphic to $\Gamma$.  
Let $M$ be a $\La_{\mathcal{O}}(G)$-module which is finitely generated over 
$\La_{\mathcal{O}}(H)$. 
\par 
Then there exists a continuous character $\rho: \Gamma \lra \Z_p^\times$ such that the largest $U$-coinvariant quotient $M(\rho)_U$ of  $M(\rho)$
is finite for every open normal subgroup $U$ of $G$. 
\end{maintheorem}
We give some  examples of a pair $H\subset G$ satisfying the condition (G) and a 
$\La_{\mathcal{O}} (G)$-module $M$ which should appear in arithmetic applications. 
\begin{example}
\begin{enumerate} 
\item 
Let us choose a prime $p \geq 5$. Let $E$ be a non-CM  elliptic curve over $\Q$ with good ordinary reduction at  $p$. Take $K = \Q(E[p]) $ and 
 set $K_\infty = \Q( {\underset{n \geq1}\cup} E[p^n])$.  Then by a well known result of Serre, $\mathrm{Gal}(K_\infty/K)$ is an open subgroup of $GL_2(\Z_p)$. 
 By Weil pairing, the cyclotomic $\Z_p$ extension $K_\cyc$ of $K$ is contained in $K_\infty$. 
We denote $\mathrm{Gal}(K_\infty/K)$, $\mathrm{Gal}(K_\infty /K_\cyc )$ and $\mathrm{Gal}(K_\cyc/K)$ by $G$, $H$ and $\Gamma$ respectively. The pair $H\subset G$ satisfies the condition \rm{(G)}. \\
Let us consider the Pontryagin dual $\mathcal{S}^\vee_A$ of the Selmer group $\mathcal{S}_A$ of the Galois representation 
$A=T_pE\otimes \Q_p /\Z_p$ over the Galois extension $K_\infty /K$ discussed above. 
We take $M$ to be the module $\mathcal{S}^\vee_A /\mathcal{S}^\vee_A(p)$.  It is conjectured that the module $M=\mathcal{S}^\vee_A /\mathcal{S}^\vee_A(p)$ is in the category $\mathfrak n_H(G)$ $($cf. \cite[Conjecture 5.1]{cfksv}$)$ and there are examples  where this conjecture is satisfied (cf. loc. cit.). 
\item Let us choose a $p$-th power free integer $m \geq 2$. Put $ K =\Q(\mu_p)$, $K_\cyc = \Q(\mu_{p^\infty})$ and 
 $K_\infty = \overset{\infty} {\underset{n=1}\cup} ~K_\cyc ( m^{1/p^n})$.  Such an extension $K_\infty/K $ is called a false-Tate curve extension.
 We denote $\mathrm{Gal}(K_\infty/K)$, $\mathrm{Gal}(K_\infty /K_\cyc )$ and $\mathrm{Gal}(K_\cyc/K)$ by $G$, $H$ and $\Gamma$ respectively. 
 Note that we have $G \cong \Z_p \rtimes \Z_p$, $H \cong \Z_p$ and $\Gamma \cong \Z_p$. Again the pair $H\subset G$ satisfies the condition (G). 
\par 
Let us consider the Pontryagin dual $\mathcal{S}^\vee_A$ of the Selmer group $\mathcal{S}_A$ of the Galois representation 
$A=T\otimes \Q_p /\Z_p$ over a Galois extension $K_\infty /K$ discussed above.  We take $M$ to be the module $\mathcal{S}^\vee_A /\mathcal{S}^\vee_A(p)$. Under certain 
assumptions on $ A$, it is expected that $\mathcal{S}^\vee_A /\mathcal{S}^\vee_A(p)$ will be in $\mathfrak n_H(G)$. We refer to \cite{hv} for some examples of 
$\mathcal{S}^\vee_A /\mathcal{S}^\vee_A(p)$ which are in $\mathfrak n_H(G)$.
\item Let $K$ be an imaginary quadratic field in which a rational prime $p \neq 2$ splits. Let $K_\infty$ be the unique 
$\Z_p^{\oplus 2} $-extension of $K$.  Let $G = \mathrm{Gal} (K_\infty /K) $ and  $H = \mathrm{Gal} (K_\infty /K_{\mathrm{cyc}}) $. Once again the pair $H \subset G$ satisfies the condition \rm{(G)}. 
\par 
For the Pontrjagin dual $\mathcal{S}^\vee_A$ of the Selmer group $\mathcal{S}_A$ of the Galois representation 
$A=T\otimes \Q_p /\Z_p$ over a Galois extension $K_\infty /\mathbb{Q}$ in this commutative two-variable situation, similar phenomena as above are expected and 
we take $M$ to be the module $\mathcal{S}^\vee_A /\mathcal{S}^\vee_A(p)$. \end{enumerate}
\end{example}
In a forthcoming joint work of two of us \cite{jo}, Main Theorem above will be applied to establish functional equation of Selmer groups for general $p$-adic representations over general non-commutative $p$-adic Lie extension. 
This is a partial motivation of our present work for two of us. 
Note that the third author proved the functional equation of Selmer groups for elliptic curves over false-Tate curve extension (cf. \cite{za1}) and for 
non CM elliptic curves in $GL_2$-extension (cf. \cite{za2}). 
But the main method of the papers \cite{za1}, \cite{za2} are not based on the twisting lemma.
\ \\ \\ 
{\bf Notation} 
Unless otherwise specified, all modules over $\La_{\mathcal{O}}( G)$ are considered as left modules. 
Throughout the paper we fix a topological generator $\gamma$ of $\Gamma$. 
\\ \\ 
{\bf Acknowledgement} 
A part of the paper was done when S. Jha  visited Osaka University. He thanks Osaka  
University for their hospitality. 
The paper was finalized when the T. Ochiai stayed at Indian Institute of Technology Kanpur (IITK) on September 2015. He thanks  IITK for their hospitality. 
We are grateful to Prof. Coates for encouragement and invaluable suggestion regarding improving an earlier draft of the paper. 
\section{Preliminary Theorem}\label{s3}
In this section, we formulate and prove  Preliminary Theorem below, which gives 
the same conclusion  as Main Theorem under stronger assumptions (i.e. the hypothesis (H) and 
non-existence of non-trivial element of order $p$ in $G$). 
In the next section, our Main Theorem is the deduced from Preliminary theorem 
and Key Lemma which is given in the next section. 
\begin{preliminarytheorem}\label{prein} 
Let $G$ be a compact $p$-adic Lie group without any element of order $p$ and let $H$ 
be its closed normal subgroup such that $G/H$ is isomorphic to $\Gamma$.  
Let $M$ be a finitely generated torsion $\La_{\mathcal{O}}(G)$-module satisfying the following condition: 
\begin{enumerate} 
\item[(H)] 
We have a $\La_{\mathcal{O}}(H)$-linear homomorphism $M {\longrightarrow} \Z_p [[H]]^{\oplus d} $ 
which induces an isomorphism $M \otimes_{\mathbb{Z}_p} \mathbb{Q}_p \overset{\sim}{\longrightarrow} 
(\Z_p [[H]] \otimes_{\mathbb{Z}_p} \mathbb{Q}_p )^{\oplus d} $ after taking $\otimes_{\mathbb{Z}_p} \mathbb{Q}_p$. 
\end{enumerate}
Then there exists a continuous character $\rho: \Gamma \lra \Z_p^\times$ such that the largest $U$-coinvariant quotient $M(\rho)_U$ of  $M(\rho)$
is finite for every open normal subgroup $U$ of $G$. 
\end{preliminarytheorem}
Before going into the proof of Preliminary Theorem, we collect some basic results in non-commutative Iwasawa theory which are relevant  for the article.
\begin{lemma}\label{lemma;l1}
Let $H\subset G$ be a pair satisfying the condition (G)  
and let $M$ be a finitely generated $\La_{\mathcal{O}} (G)$-module which satisfies the condition (H). 
Then there exists a matrix $A \in M_d (\Z_p [[H]] \otimes_{\Z_p} \mathbb{Q}_p )$ such that   
$M \otimes_{\Z_p} \mathbb{Q}_p $ is isomorphic to 
\begin{equation}\label{equation;e1}
( \Z_p [[G]] \otimes_{\Z_p} \mathbb{Q}_p )^{\oplus d} \big/ ( \Z_p [[G]] \otimes_{\Z_p} \mathbb{Q}_p )^{\oplus d} 
(\widetilde{\gamma} \text{\boldmath$1$}_d -A )
\end{equation}
where $\gamma$ is a topological generator of $\Gamma$ and $\widetilde{\gamma} \in G$ is a fixed lift of $\gamma$ 
and elements in $( \Z_p [[G]] \otimes_{\Z_p} \mathbb{Q}_p )^{\oplus d} $  are regarded as row vectors. 
\end{lemma}
\begin{proof}
Let us take a basis $\bold{v}_1 ,\ldots ,\bold{v}_d $ of free  $\Z_p [[H]] \otimes_{\Z_p} \mathbb{Q}_p$-module $M \otimes_{\Z_p} \mathbb{Q}_p$. 
Through the isomorphism $M \otimes_{\mathbb{Z}_p} \mathbb{Q}_p \overset{\sim}{\longrightarrow} 
(\Z_p [[H]] \otimes_{\mathbb{Z}_p} \mathbb{Q}_p )^{\oplus d} $ fixed by the condition (H), $\widetilde{\gamma} $ acts on  $M$. 
Thus we define a matrix $A = (a_{ij})_{1\leq i,j \leq d} \in M_d (\Z_p [[H]] \otimes_{\Z_p} \mathbb{Q}_p )$ by 
$$
\widetilde{\gamma} \cdot \bold{v}_i = \sum_{1\leq j \leq d} a_{ji}  \bold{v}_j . 
$$ 
We denote the module presented as in \eqref{equation;e1} by $N_A$. By construction, we have a $\Z_p [[H]] \otimes_{\Z_p} \mathbb{Q}_p $-linear isomorphism   
$( \Z_p [[H]] \otimes_{\Z_p} \mathbb{Q}_p )^{\oplus d} \overset{\sim}{\longrightarrow} N_A$ 
on which $\widetilde{\gamma} $ acts in the same manner as the action of $\widetilde{\gamma} $ on $M  \otimes_{\Z_p} \mathbb{Q}_p$.
\end{proof}
We denote by $\mathcal{U}$ the set of all open normal subgroups $U$ of $G$. We remark that the set $\mathcal{U}$ is a countable set 
since $G$ is profinite and has a countable base at identity.  

\begin{lemma}\label{lemma;l2}
For any $U \in \mathcal{U}$, $\Z_p [G /U ] \otimes_{\Z_p } \overline{\mathbb{Q}}_p$ is isomprphic to 
a finite number of products of matrix algebras $\prod^{k (U)}_{i=1} M_{r_i} (\overline{\mathbb{Q}}_p)$. 
\end{lemma}
\begin{proof}
First of all, the algebra $\Z_p [G /U ] \otimes_{\Z_p }{\mathbb{Q}}_p \cong \mathbb{Q}_p [G /U ]$ is a 
semisimple algebra over $\mathbb{Q}_p$ since $G/U$ is a finite group and $\mathbb{Q}_p$ is of characteristic $0$. 
We have an isomorphism 
$$
\Q_p [G /U ] \cong \prod^{l_n}_{i=1} M_{s_i} (D_i )  
$$ 
where $D_i$ is a finite dimensional division algebra over $\mathbb{Q}_p$. 
For each $i$, the center $K_i$ of $D_i$ is a finite extension of $\mathbb{Q}_p$. It is well-known that 
$\mathrm{dim}_{K_i } D_i$ is a square of some natural number $t_i$ and  $D_i   \otimes_{K_i } \overline{\mathbb{Q}}_p$ 
is isomorphic to $M_{t_i} (\overline{\mathbb{Q}}_p)$. Thus $M_{s_i} (D_i )  \otimes \overline{\mathbb{Q}}_p$ is isomorphic to 
$[K_i : \mathbb{Q}_p]$ copies of  $M_{s_i + t_i} (\overline{\mathbb{Q}}_p) $.  The lemma follows immediately from this. 
\end{proof}
We will prove Preliminary Theorem using the results prepared above. 
\begin{proof}[Proof of Preliminary Theorem]
First, we remark that,  for an open normal subgroup $U$ of $G$, we have 
\begin{equation}\label{equation;e2_0}
\text{$M(\rho)_U$ is finite if and only if 
$M(\rho)_U \otimes_{\Z_p} {\mathbb{Q}}_p =0$.} 
\end{equation}
Since taking the base extension $\otimes_{\Z_p} {\mathbb{Q}}_p$ and 
taking the largest $U$-coinvariant quotient commute with each other, by Lemma \ref{lemma;l1}, we have  
\begin{equation}\label{equation:e2_1}
M(\rho)_U \otimes_{\Z_p} {\mathbb{Q}}_p \cong ( \Z_p [G/U] \otimes_{\Z_p} \mathbb{Q}_p )^{\oplus d} \big/ ( \Z_p [G/U] \otimes_{\Z_p} \mathbb{Q}_p )^{\oplus d} 
(\widetilde{\gamma}^{\oplus d} _U  -A_U (\rho ) )
\end{equation}
where we denote the projection of $\widetilde{\gamma}  \in G$ to $G/U$ by $\widetilde{\gamma}_U $ .  
Here, the matrix $A_U (\rho) \in M_d ( \Z_p [G/U]\otimes_{\Z_p} {\mathbb{Q}}_p)$  is defined to be the image of $\rho (\gamma)^{-1} A\in M_d ( \Z_p [[H]] \otimes_{\Z_p} {\mathbb{Q}}_p) $ via the composite map 
$M_d ( \Z_p [[H]]\otimes_{\Z_p} {\mathbb{Q}}_p) \longrightarrow M_d ( \Z_p [[G]]\otimes_{\Z_p} {\mathbb{Q}}_p) \longrightarrow 
M_d ( \Z_p [G/U]\otimes_{\Z_p} {\mathbb{Q}}_p) $. Taking the base extension $\otimes_{\mathbb{Q}_p} \overline{\mathbb{Q}}_p $
of the isomorphism \eqref{equation:e2_1}, we have a $\overline{\mathbb{Q}}_p$-linear isomorphism by Lemma \ref{lemma;l2}: 
\begin{equation}\label{equation:e2_2}
M(\rho)_U \otimes_{\Z_p} \overline{\mathbb{Q}}_p  \cong \prod^{k (U)}_{i=1} M_{r_i} (\overline{\mathbb{Q}}_p)^{\oplus d}
\Big/
M_{r_i} (\overline{\mathbb{Q}}_p)^{\oplus d} (\gamma_{U,i}^{\oplus d}  - A_{U,i} (\rho ) )
\end{equation}
where $\gamma_{U,i}  \in \mathrm{Aut}_{\overline{\mathbb{Q}}_p} (M_{r_i} (\overline{\mathbb{Q}}_p))$ is defined as follows. 
We consider the base extension to $\overline{\mathbb{Q}}_p$ of 
$\widetilde{\gamma}_U  \in  \mathrm{Aut}_{\Z_p [G/U] \otimes_{\Z_p} \mathbb{Q}_p} (\Z_p [G/U] \otimes_{\Z_p} \mathbb{Q}_p ) 
\subset \mathrm{Aut}_{\mathbb{Q}_p} (\Z_p [G/U] \otimes_{\Z_p} \mathbb{Q}_p ) $. 
This is an element of $\mathrm{Aut}_{\overline{\mathbb{Q}}_p} (\prod^{k (U)}_{i=1} M_{r_i} (\overline{\mathbb{Q}}_p))$. 
Then, we denote the projection of this element to the $i$-th component by $\gamma_{U,i}$.
\par 
The element $A_{U,i} (\rho ) \in \mathrm{End}_{\overline{\mathbb{Q}}_p} (M_{r_i} (\overline{\mathbb{Q}}_p)^{\oplus d})$ 
is defined as follows. We consider the base extension to $\overline{\mathbb{Q}}_p$ of 
$
A_U (\rho) \in M_d ( \Z_p [G/U]\otimes_{\Z_p} {\mathbb{Q}}_p)
\subset \mathrm{End}_{ {\mathbb{Q}}_p}  \big( (  \Z_p [G/U]\otimes_{\Z_p} {\mathbb{Q}}_p)^{\oplus d} \big) .
$ 
This is an element of $\mathrm{End}_{\overline{\mathbb{Q}}_p} (\prod^{k (U)}_{i=1} M_{r_i} (\overline{\mathbb{Q}}_p)^{\oplus d})$. 
 Then, we denote the projection of this element to the $i$-th component by $A_{U,i} (\rho)$.
 \par 
Now, we denote by  $A_{U,i} $ the element $A_{U,i} (\text{\boldmath$1$})$. We remark that $A_{U,i} (\rho)$ is equal to 
$\rho (\gamma)^{-1}A_{U,i}$ for any continuous character $\rho: \Gamma \longrightarrow \Z^\times_p$. 
we define $\mathrm{EV}_{U,i}$ to be the set of roots of the characteristic polynomial 
$$
P_{U,i} (T) := \mathrm{det} (\gamma_{U,i}^{\oplus d}  - A_{U,i} T ). 
$$
Since $\gamma_{U,i}^{\oplus d}$ is automorphism, the polynomial $P_{U,i} (T)$ is not zero.  Hence $\mathrm{EV}_{U,i}$ is a finite set. 
We denote the union of $\mathrm{EV}_{U,i}$ for $1\leq i \leq k(U)$ by $\mathrm{EV}_{U}$, which is again a finite set. 
If $\rho (\gamma)^{-1}$ is not contained in $\mathrm{EV}_{U} \cap \Z^\times_p$, the module in \eqref{equation:e2_2} is zero 
and hence the module in \eqref{equation:e2_1} is zero. 
Now, we denote by $\mathrm{EV}_{M}$ the union  of $\mathrm{EV}_{U}\cap \Z^\times_p$ when $U \in \mathcal{U}$ moves. 
Since $\mathcal{U}$ is countable set, $\mathrm{EV}_{M}$ is a countable set. Thus $\Z^\times_p \setminus \mathrm{EV}_{M}$ 
is nonempty since $\Z^\times_p $ is uncountable. By choosing $\rho (\gamma)^{-1} \in \Z^\times_p \setminus \mathrm{EV}_{M}$, 
we complete the proof. 
\end{proof}
\section{Proof of the main theorem}\label{s4}
In this section, we prove Main Theorem. 
First we prepare Key Lemma as follows: 
\begin{keylemma}
Let $G$ be a compact $p$-adic Lie group without any element of order $p$ and let $H$ 
be its closed subgroup such that $G/H$ is isomorphic to $\Gamma$.  
Let $M$ be $\La_{\mathcal{O}}(G)$-module which is finitely generated over 
$\La_{\mathcal{O}}(H)$. 
Then, there exists an open subgroup $G_0 \subset G$ containing $H$, 
a $\La_{\mathcal{O}}(G_0)$-module $N$ which 
is a free $\La_{\mathcal{O}}(H)$-module of finite rank, and 
a surjective $\La_{\mathcal{O}}(G_0)$-linear homomorphism $N \twoheadrightarrow M$.  
\end{keylemma}
\begin{proof}
We denote by $I$ the Jacobson radical of 
$\La_{\mathcal{O}}(H)$. Note that I is a two sided ideal of 
$\La_{\mathcal{O}}(H)$ such that we have $\La_{\mathcal{O}}(H)/I 
\cong \mathbb{F}_q$ where $\mathbb{F}_q$ is 
the residue field of $\mathcal{O}$. We also have 
$\La_{\mathcal{O}}(G)/I \cong \mathbb{F}_q [[\Gamma ]]$ by definition. 
\par 
Let us take a system of generators $m_1 ,\ldots ,m_d$ of 
$M$ as a $\La_{\mathcal{O}}(H)$-module. 
Note that $M$ is equipped with a topology obtained by a natural 
$\La_{\mathcal{O}}(H)$-module structure. 
The set $\{ I^n M\}_{n\in \mathbb{N}}$ forms a system of open neighborhoods 
of $M$. 
\par 
Choose a topological generator $\gamma$ of $\Gamma$ and 
take a lift $\widetilde{\gamma} \in G$ of $\gamma$. 
By continuity of the action of $G$ on $M$, the following two conditions  
hold true simultaneously for a sufficiently large integer $n$. 
\begin{enumerate} 
\item[(i)] 
We have $({\widetilde{\gamma}}^{p^n} -1)m_i \in IM$ for any $i$ with 
$1\leq i \leq d$. 
\item[(ii)]
The conjugate action of ${\widetilde{\gamma}}^{p^n}$ on $I/I^2$ is 
trivial. 
\end{enumerate}
We will choose and fix a sufficiently natural number $n$ satisfying the 
conditions (i) and (ii). Then we define $G_0$ to be the preimage of 
$\Gamma^{p^n}$ by the surjection $G \twoheadrightarrow \Gamma$.  
By definition, $G_0$ an open subgroup of $G$ which contain $H$.
\par 
Let us consider the set 
$\{ a_{ij } \in I\}_{1\leq i,j \leq d}$ such that 
we have $({\widetilde{\gamma}}^{p^n} -1)m_j = \sum^d_{i=1} a_{ij} m_i$.  
We consider $F$ (resp. $F'$) which is a 
free $\La_{\mathcal{O}}(G_0)$-module of rank $d$ equipped with 
a system of generators $f_1 ,\ldots ,f_d$ (resp. $f'_1 ,\ldots ,f'_d$). 
We consider a $\La_{\mathcal{O}}(G_0)$-linear homomorphism $\varphi : F' \longrightarrow F$ as follows: 
$$
\varphi : \ F' \longrightarrow F ,\ \  f'_j \mapsto ({\widetilde{\gamma}}^{p^n} -1)f_j 
- \sum^d_{i=1} a_{ij} f_i .
$$ 
We define a $\La_{\mathcal{O}}(G_0)$-module $N$ to be the cokernel of the 
map $\varphi $ above. 
We have the following claim: 
\begin{claim}
For each $i$ with $1\leq i \leq d$, 
we denote the image of $f_i$ by $\overline{f}_i$. 
Then, the $\La_{\mathcal{O}}(G_0)$-module $N$ 
is a free $\La_{\mathcal{O}}(H)$-module of finite rank $d$ 
with a system of generators $\overline{f}_1 ,\ldots ,\overline{f}_d$.
\end{claim}
If the claim holds true, a $\La_{\mathcal{O}}(G_0)$-linear homomorphism $N \twoheadrightarrow M$ sending $\overline{f}_i$ to $m_i$ for each $i$ is 
surjective. Since $N$ is free over $\La_{\mathcal{O}}(H)$, this is what we want. 
\par 
We thus prove the above claim for the rest of the proof. 
Note that we have $\La_{\mathcal{O}}(G_0)/I \cong \mathbb{F}_q [[\Gamma ]]$. 
By applying the functor $\La_{\mathcal{O}}(G_0)/I \La_{\mathcal{O}}(G_0) 
\otimes_{\La_{\mathcal{O}}(G_0)} \cdot$ to the map $\varphi $, we obtain 
$$ 
\varphi_I : \ \oplus^d_{j=1} \mathbb{F}_q [[\Gamma^{p^n} ]] f'_j 
\xrightarrow{\times ({\widetilde{\gamma}}^{p^n} -1)} 
\oplus^d_{j=1} \mathbb{F}_q [[\Gamma^{p^n} ]] f_j . 
$$
Since we have $N/IN$ is isomorphic to the cokernel of the above map $\varphi_I $, 
$N/IN$ is a free $\mathbb{F}_q$-module of rank $d$. By applying 
the topological Nakayama Lemma (cf. \cite[Cor. in \S 3]{bh}) to the compact $\La_{\mathcal{O}}(H)$-module $N$, 
$N$ is generated by $\overline{f}_1 ,\ldots ,\overline{f}_d$ 
over $\La_{\mathcal{O}}(H)$. We will prove that $N$ is free of rank $d$ over 
$\La_{\mathcal{O}}(H)$ 
with this system of generators. Now let $r$ be an arbitrary natural number. 
Since we have a natural surjection from 
$r$-fold tensor product of $I/I^2$ to $I^r/I^{r+1}$, the conjugate action of 
${\widetilde{\gamma}}^{p^n}$ on $I^r /I^{r+1}$ is also trivial. Thus, 
by applying the functor $I^r/I^{r+1} \La_{\mathcal{O}}(G_0) 
\otimes_{\La_{\mathcal{O}}(G_0)} \cdot$ to the map $\varphi$,   
we obtain a $\La_{\mathcal{O}}(G_0)$-linear map 
$$
\varphi \otimes I^r /I^{r+1}:\ I^r F' /I^{r+1} F' \longrightarrow I^r F /I^{r+1} F
$$ 
which is again defined as a multiplication of $({\widetilde{\gamma}}^{p^n} -1)$. 
This proves 
\begin{align*}
\mathrm{dim}_{\mathbb{F}_q} N/I^s N & = \sum^{s-1}_{r=0}\mathrm{dim}_{\mathbb{F}_q} 
I^rN/I^{r+1} N \\ 
& = \sum^{s-1}_{r=0}\mathrm{dim}_{\mathbb{F}_q} 
(I^r /I^{r+1} )^{\oplus d} .
\end{align*}
Thus the cardinality of 
$ N/I^s N$ is equal to the cardinality of $(\La_{\mathcal{O}}(H)/I)^{\oplus d}$ 
for any natural number $s$, which implies that $N$ is free of rank $d$ 
over $\La_{\mathcal{O}}(H)$. This completes the proof of the claim.
\end{proof}
Using Key Lemma and Preliminary Theorem, we will prove our Main Theorem. 
\begin{proof}[Proof of Main Theorem]
We will prove the main theorem by two step arguments. 
\par 
First, we consider the situation where $G$ is a compact $p$-adic Lie group 
without any element of order $p$ and $H$ 
is its closed subgroup such that $G/H$ is isomorphic to $\Gamma$. 
Thus we dropped the assumption (H) of Preliminary Theorem but we 
still keep the assumption of non-existence of non-trivial element of order $p$ in $G$. 
\par 
Let $M$ be $\La_{\mathcal{O}}(G)$-module which is finitely generated over 
$\La_{\mathcal{O}}(H)$. 
By Preliminary Theorem, for a sufficiently large natural number $n$,  
we have a surjective $\La_{\mathcal{O}}(G_0)$-linear homomorphism 
$N \twoheadrightarrow M$ from a free $\La_{\mathcal{O}}(H)$-module of finite rank. 
Here $G_0$ is a unique open subgroup $G_0 \subset G$ of index $p^n$ 
containing $H$. Note that the module $N$ satisfies the condition (H) of Preliminary Theorem. 
We thus find a continuous 
character $\rho_0 : \Gamma^{p^n} \longrightarrow \mathbb{Z}_p^\times$ such that 
$N (\rho_0)_{U_0}$ is finite for any open normal subgroup $U_0$ of $G_0$. 
By the proof of Preliminary Theorem, we can choose uncountably many 
such $\rho_0$. Thus, we see that we can take $\rho_0$ as above 
so that the value of $\rho_0$ is contained in a open subgroup 
$1+p^n \mathbb{Z}_p$ of $\mathbb{Z}^\times_p$. Then, we take a 
continuous character $\rho : \Gamma \longrightarrow \mathbb{Z}_p^\times$ 
whose restriction to $\Gamma^{p^n}$ 
coincides with $\rho_0$. The twist $M(\rho)$ with this character is what 
we want in our Main Theorem. In fact, for any open normal subgroup $U$ of $G$, 
we have a surjection $N(\rho_0)_{U_0} \twoheadrightarrow M(\rho)_U$ 
taking an open normal subgroup $U_0$ of $G_0$ contained in $U$. 
Since $N(\rho_0)_{U_0}$ is finite by Preliminary Theorem, 
$M(\rho)_U$ must be finite. Thus we finished the proof of our Main Theorem 
under the assumption of non-existence of non-trivial element of order $p$ in $G$.
\par 
Now, we deduce our Main Theorem assuming that it is true 
under the assumption of non-existence of non-trivial element of order $p$ in $G$.  
We consider the situation where $G$ is a compact $p$-adic Lie group 
with elements of order $p$ and $H$ is its closed subgroup such that 
$G/H$ is isomorphic to $\Gamma$. 
Let $M$ be $\La_{\mathcal{O}}(G)$-module which is finitely generated over 
$\La_{\mathcal{O}}(H)$. Let $G'$ be   a uniform open normal subgroup of $G$ (cf. 
\cite[Chap. III, \S (3.1)]{la}), 
which is automatically without any elements of order $p$. Let $H'$ be the intersection of 
$H$ and $G'$. Since $M$ is finitely generated over $\La_{\mathcal{O}}(H)$ 
and since $H'$ is of finite index in $H$, $M$ is finitely generated over 
$\La_{\mathcal{O}}(H')$. According to the result in our first step, 
there exist a continuous character $\rho':G'/H'\longrightarrow \mathbb{Z}_p^\times$ 
such that $M(\rho' )_{U'}$ 
is finite for every open subgroup $U'$ of $G'$. 
Note that $G'/H'$ is naturally regarded as an open subgroup of $G/H$. 
Thus by choosing $\rho'$ so that the image of $\rho'$ is enough small in 
$\mathbb{Z}^\times_p$ compared to the index of $G'/H'$ in $G/H$, 
there exists a continuous character $\rho : G/H \longrightarrow \mathbb{Z}_p^\times$ 
whose restriction to $G'/H'$ coincides with $\rho'$. 
Now, for any open normal subgroup $U$ of $G$, we take an open normal subgroup 
$U'$ of $G'$ which is contained in $U$. We have a natural map 
$M(\rho' )_{U'} \longrightarrow M(\rho )_U$ where $M(\rho' )_{U'}$ is finite 
by the choice of $\rho'$ and by our discussion above.   
Unlike as in the first step, $M(\rho' )_{U'} \longrightarrow M(\rho )_U$ 
is not necessarily surjective. However, the cokernel of this map is still 
finite by construction. We thus deduce that $M(\rho )_U$ is finite. 
This completes the proof of Main Theorem.
\end{proof}
\begin{rem}[$p$-torsion modules]
For a compact $p$-adic Lie group $G$ without any element of order $p$, it is well-known that $\La_O(G)$ is left and right noetherian. 
Let $N$ be a finitely generated $p$-primary torsion left $\La_{\mathcal O}(G)$ module. Then, we have $N = N[p^r]$ for some $r \in \N$. For any open normal  subgroup $U$ of $G$, 
$N_U$ is a finitely generated $\Z/{p^r\Z}[G/U]$ module. In other words, $N_U$ is  always finite when $N$ is of $p$-primary torsion. 
\par 
For a finitely generated torsion $\La_{\mathcal O}(G)$ module $M$,  we consider the exact sequence
$$ 0 \lra M(p) \lra M \lra M/M(p) \lra 0,$$ 
where $M(p)$ is the largest $p$-primary torsion submodule of $M$. 
Then, from the  preceding discussion, it is clear  that in the situation of  Main Theorem, for any continuous $\rho: \Gamma \lra \Z_p^\times$ and 
for any open normal subgroup $U$ of $G$, $M(\rho)_U$ is finite if and only if $  \left( \dfrac{M}{M(p)}(\rho) \right)_U$ is finite. 
\par
In particular, when we want to apply the result in Main Theorem to arithmetic situations  coming from Selmer groups of 
certain Galois module $A$, we remark that $\mathcal{S}^\vee_A(\rho)_U$ is finite if and only if $\left( \dfrac{\mathcal{S}^\vee_A}{\mathcal{S}^\vee_A(p)}(\rho) \right)_U$ is finite.
\end{rem}

\end{document}